\documentclass[12pt,a4paper]{amsart}

\usepackage{amsfonts}

\usepackage{amsthm}

\usepackage{amsmath}

\usepackage{amscd}

\usepackage[latin2]{inputenc}

\usepackage{t1enc}

\usepackage[mathscr]{eucal}

\usepackage{indentfirst}

\usepackage{graphicx}

\usepackage{graphics}

\usepackage{pict2e}

\numberwithin{equation}{section}

\theoremstyle{plain}

\newtheorem{Th}{Theorem}[section]

\newtheorem{Lemma}[Th]{Lemma}

\newtheorem{Cor}[Th]{Corollary}

\newtheorem{Pro}[Th]{Proposition}

\newtheorem{Fact}[Th]{Fact}

 \theoremstyle{definition}

\newtheorem{Def}[Th]{Definition}

\newtheorem{Rem}[Th]{Remark}

\newtheorem{?}[Th]{Problem}

\newcommand{\N}{\mathbb{N}}

\newcommand{\C}{\mathbb{C}}

\begin{document}

\title[Root moments  of graph polynomials]{Benjamini--Schramm continuity of 
root moments  of graph polynomials}

\author[P. Csikv\'ari]{P\'{e}ter Csikv\'{a}ri}

\address{Massachusetts Institute of Technology \\ Department of Mathematics \\
Cambridge MA 02139 \&  E\"{o}tv\"{o}s Lor\'{a}nd University \\ Department of Computer 
Science \\ H-1117 Budapest
\\ P\'{a}zm\'{a}ny P\'{e}ter s\'{e}t\'{a}ny 1/C \\ Hungary} 

\email{peter.csikvari@gmail.com}

\author[P. E. Frenkel]{P\'eter E. Frenkel}

\address{E\"{o}tv\"{o}s Lor\'{a}nd University \\ Department of Algebra and 
Number Theory \\ H-1117 Budapest
\\ P\'{a}zm\'{a}ny P\'{e}ter s\'{e}t\'{a}ny 1/C \\ Hungary}

\email{frenkelp@cs.elte.hu}

\thanks{}

\thanks{ Both authors are partially supported by MTA R\'enyi
"Lend\"ulet" Groups and Graphs Research Group.}



\begin{abstract} Recently, M.\ Ab\'ert and T.\ Hubai studied the following
problem. The chromatic measure of a finite simple graph is defined to be the 
uniform distribution on its chromatic roots.   Ab\'ert and  Hubai proved that
for a Benjamini-Schramm convergent sequence of finite graphs, the chromatic
measures converge in holomorphic moments. They also showed that the
normalized logarithm of the chromatic polynomial converges to a harmonic real
function outside a bounded disc. 

In this paper we generalize their work to a  wide class of graph polynomials,
namely, multiplicative graph polynomials of bounded exponential
type.   A special case of our results is  that for any fixed complex
number $v_0$ the measures arising from the  Tutte polynomial $Z_{G_n}(z,v_0)$
converge in holomorphic moments if the sequence  $(G_n)$ of finite graphs is
Benjamini--Schramm convergent. This answers a question of   Ab\'ert and  Hubai
in the affirmative. Even in the original case of the chromatic polynomial, our
proof is considerably simpler. 
\end{abstract}

\maketitle

\section{Introduction --- background and main results} 
This paper generalizes certain results of Sokal \cite{sok}, Borgs, Chayes,
Kahn and Lov\'asz \cite{bckl} and, most importantly, Ab\'ert and Hubai
\cite{abhu}. In this section we briefly recall these results and state our
main theorems: \ref{f-Sokal} and \ref{main}.

Let $G=(V,E)$ be a finite, simple, undirected graph. (Note that we follow the
usual notations. However, if some notation is unclear, the reader may find its
meaning at the end of this section.) A map $f: V\rightarrow \{1,2,\dots ,q\}$
is a \it proper coloring \rm if for all edges $(x,y)\in E$ we have $f(x)\neq
f(y)$. For a positive integer $q$ let $ch(G,q)$ denote the number of proper
colorings of $G$ with $q$ colors. Then it turns out that $ch(G,q)$ is a
polynomial in $q$. It is called the chromatic polynomial \cite{rea}. Let us
call the roots of  the chromatic polynomial \textit{chromatic roots}. The
\textit{chromatic measure}  is the probability measure  $\mu_G$ on
$\mathbb{C}$ given by the uniform distribution on the chromatic
roots. 

\subsection{The Sokal bound}
More than a decade ago, Alan Sokal proved the following theorem.

\begin{Th}[Sokal \textrm{\cite{sok}}]\label{Sokal} Let $G$ be a graph of largest
degree  $\Delta$. Then the absolute value of any root of the  chromatic
polynomial of $G$ is at most $C\Delta$, where the constant $C$ is less than
$8$.   
\end{Th} 

Later Jackson, Procacci and Sokal \cite{jps} extended this result to a more
general graph polynomial, the Tutte polynomial.  In Theorem~\ref{f-Sokal}, we
shall further generalize Theorem~\ref{Sokal} to a wide class of graph
polynomials. This will set  the stage for our main result, Theorem~\ref{main}.

\begin{Def}\label{graphpol} A \textit{graph polynomial} is a mapping $f$ that
  assigns to every  finite simple graph $G$ a polynomial $f(G,x)\in\C[x]$. The
  graph polynomial  $f$ is \textit{monic} if the polynomial $f(G,x)$ is monic
  of degree $|V(G)|$ for all graphs $G$. 
\end{Def}

\begin{Def}\label{exponential} We say that the graph polynomial $f$ is
  \textit{of exponential  type} if  $f(\emptyset,x)=1$  and for every graph
  $G=(V(G),E(G))$ we have  
\begin{equation}
\sum_{S\subseteq V(G)}f(S,x)f(G- S,y)=f(G,x+y),
\end{equation}
where $f(S,x)=f(G[S],x)$ and  $f(G- S,y)=f(G[V(G)\setminus S],y)$ are the 
polynomials of the subgraphs of $G$ induced by the sets $S$ and 
$V(G)\setminus S$, respectively.
\end{Def}

In the next  definition, we introduce a boundedness condition that is
relatively easy to check and  implies a  Sokal-type bound on the  roots.

\begin{Def} \label{cond-exp} Assume that 
$$f(G,x)=\sum_{k=0}^n a_k(G)x^k$$ 
is a monic graph polynomial of exponential type. Assume that there
is a function $R:\N\to[0,\infty)$ not depending on $G$ such that for any graph
$G$ with all  degrees at most $\Delta$, any vertex $v\in V(G)$ and any 
$t\ge 1$ we have
$$\sum_{v\in S\subseteq V(G) \atop |S|=t}|a_1(G[S])|\leq R(\Delta)^{t-1}.$$
In this case we say that the graph polynomial $f(G,x)$ is  \textit{of bounded 
exponential type} with bounding function $R$.  
\end{Def}

\begin{Rem}\label{examples} Examples of  graph polynomials of bounded
  exponential type include the chromatic polynomial, the Tutte polynomial, the
  (modified) matching  polynomial, the adjoint polynomial and the Laplacian
  characteristic polynomial. Cf.\ Remark~\ref{ex} and Proposition~\ref{bounded}.
\end{Rem}

By imitating Sokal's proof, we will prove the following extension of his
theorem  (note, however, that our constant is much weaker).

\begin{Th} \label{f-Sokal} Let $f(G,x)$ be a  graph polynomial of bounded
exponential type with bounding function $R>0$. Let $G$ be a graph of largest
degree  $\le\Delta$. Then the absolute value of any root of $f(G,z)$ is less
than $cR(\Delta)$, where $c<7.04$.   
\end{Th} 

\begin{Def} We say that the graph polynomial  $f$ \textit{has bounded roots}
 if there exists a function $\tilde R:\N\to (0,\infty)$  such that for every
 $\Delta\in\N$ and every graph $G$ with all degrees at most $\Delta$, the 
roots of the polynomial $f(G,z)$ have absolute value less than $\tilde
R(\Delta)$. 
%In this case we also say that $f $  \textit{has roots bounded by}
%the function $\tilde R$.
\end{Def}

According to Sokal's Theorem~\ref{Sokal}, the chromatic polynomial has
bounded roots.  By Theorem~\ref{f-Sokal},  every graph polynomial of bounded
exponential type has bounded roots.   

\subsection{Convergence}
Recently, M.\ Ab\'ert and T.\ Hubai \cite{abhu} studied the  behaviour
of chromatic measures in Benjamini--Schramm convergent graph sequences.

We recall the definition of Benjamini--Schramm convergence.
For a finite graph $G$, a finite rooted graph $\alpha$ and a positive integer
$r$, let $\mathbb{P}(G,\alpha,r)$ be the probability that the $r$-ball
centered at a uniform random vertex of $G$ is isomorphic to $\alpha$. We say
that a graph sequence $(G_n)$ of bounded degree is Benjamini-Schramm
convergent if for all finite rooted graphs $\alpha$ and $r>0$, the
probabilities $\mathbb{P}(G_n,\alpha,r)$ converge. This means that one cannot
distinguish $G_n$ and $G_{n'}$ for large $n$ and $n'$ by sampling them from a
random vertex with a fixed radius of sight.

Our main interest is in the behaviour of the roots of $f(G_n,z)$ as
$n\to\infty$. For our main result, we need to put  two further restrictions
on $f$. 

\begin{Def}\label{isom-inv}
The graph polynomial $f$ is
\textit{isomorphism-invariant} if   $G_1\simeq G_2$ implies 
$f(G_1,x)=f(G_2,x)$.
\end{Def}

\begin{Def}\label{multiplicativity}
The graph polynomial $f$ is
\textit{multiplicative} if   
$$f(G_1\uplus G_2,x)=f(G_1,x)f(G_2,x),$$
where $G_1\uplus G_2$ denotes the disjoint union of the graphs $G_1$ and $G_2$.
\end{Def}

The examples in Remark~\ref{examples} are all isomorhism-invariant and
multiplicative (in addition to  being monic, of exponential type, and having
bounded roots). 

Now we are ready to state our main result.

\begin{Th} \label{main} Let $f$ be an isomorphism-invariant monic
multiplicative graph polynomial of exponential type. Assume that $f$ has
bounded roots.

Let $(G_n)$ be a Benjamini-Schramm convergent graph sequence.
Let $K\subset\mathbb C$ be a compact set  containing all roots of $f(G_n,x)$
for all $n$, such that $\mathbb C\setminus K$ is connected.

\begin{itemize}
\item[(a)] For a graph $G$, let $\mu_{G}$ be the uniform distribution on
the roots of $f(G,x)$. Then for every continuous function
$g:K\rightarrow \mathbb{R}$ that is harmonic on the interior of $K$, the
sequence
$$\int_Kg(z)d\mu_{G_n}(z)$$
converges.

Moreover, for any open set $\Omega\subseteq \mathbb R^d$ and any continuous
function
$g:K\times \Omega\rightarrow \mathbb{R}$ that is  harmonic on
the interior of $K$ for any fixed $\xi\in \Omega$ and harmonic on $\Omega$
for any fixed $z\in K$, the sequence
$$\int_Kg(z,\xi)d\mu_{G_n}(z)$$
converges, locally uniformly in $\xi\in\Omega$, to a harmonic function on
$\Omega$.

\item[(b)] For  $\xi\in
 \C\setminus  K$, let
$$t_{n}(\xi)=\frac{\log | f(G_n,\xi)|}{|V(G_n)|}.$$
Then $t_{n}(\xi)$ converges to a harmonic function locally
uniformly on  $\C\setminus K$.

\end{itemize}

\end{Th}

The integral in (a) is called a \it harmonic moment \rm of the roots of
$f(G_n,z)$. Its convergence is referred to as the `Benjamini--Schramm
continuity' or `estimability' of the moment.
The fraction in (b) is called the \it entropy per vertex \rm or the \it free
energy \rm at $\xi$.

The phenomenon described in the Theorem was discovered, in the case of the
chromatic polynomial,  by  M.\ Ab\'ert and T.\ Hubai. The main achievement
of their paper
\cite{abhu} was to state and prove Theorem~\ref{main} for the chromatic
polynomial (in a slightly different form: they used a disc in place of our set
$K$, they used holomorphic functions in place of our harmonic ones, they took
$\log f$ in place  of $\log |f|$ in (b) --- all these 
 are inessential differences). 
They thus answered a question of Borgs and generalized a result of Borgs, 
Chayes, Kahn and Lov\'asz \cite{bckl}, who had proved convergence of $(\log
ch(G_n,q))/|V(G_n)|$
at large 
positive integers $q$.

One may naively hope for the weak convergence of the
measures $\mu_{G_n}$ arising from the roots.  As remarked by Ab\'ert and Hubai,
this is easily seen to be false already for the chromatic polynomial. Indeed,
 paths and circuits together form a
Benjamini--Schramm convergent sequence, but the chromatic measures of paths
tend to the Dirac measure at $1$ and the chromatic measures of circuits
tend to the uniform measure on the unit circle centered at $1$.   
So it was a crucial observation that the useful relaxation is to
consider the convergence of the holomorphic moments.   

To prove  the convergence of the holomorphic moments,   Ab\'ert and  Hubai
\cite{abhu} showed  (essentially) that for a finite graph $G$ and for every
$k$, the number 
$$p_k(G)=|V(G)|\int_Kz^kd\mu_{G}(z)$$
can be expressed as a fixed linear combination of the numbers $H(G)$, where 
the $H$ are non-empty \textbf{connected} finite graphs and $H(G)$ denotes the
number of subgraphs of $G$ isomorphic to $H$. 

To show this, they determined an exact expression for the power sum $p_k(G)$
of the roots in terms of homomorphism numbers. Our approach is a bit
different: without  determining  the exact expression, we use multiplicativity
of the graph polynomial  to show that the coefficient of $H(G)$ for
non-connected $H$ must be $0$. (In fact, we can determine the exact expression
too with a little extra work: see Remark~\ref{coeff-formula}.)  The argument is
simplified by using subgraph counting, which is equivalent to injective
homomorphism numbers, instead of ordinary homomorphism numbers.

Besides seeking a proof without lengthy calculations, 
we felt  it desirable to  clarify  which properties of the
chromatic polynomial are needed to prove such a result. We wished
to grasp the right concepts for the  generalization
 of the Ab\'ert--Hubai result.  This is achieved by the definitions we have
 given in this section.

\bigskip

This paper is organized as follows. In  Section~\ref{Tutte} we recall the 
Tutte polynomial and explain how and why Theorem~\ref{main} applies to it. 
In  Section~\ref{prelim} we recall some basic facts related to subgraph 
counting. In Section~\ref{multiplicative} we prove Theorem~\ref{f-conv-entr},
which will clearly cover Theorem~\ref{main} for  the most interesting graph
polynomials, including the chromatic polynomial and the Tutte polynomial. The
fact that  Theorem~\ref{f-conv-entr} covers  Theorem~\ref{main} in general
will be shown in Section~\ref{exp}.

In Section~\ref{exp} we study  graph polynomials of exponential type. This
part will consist of two subsections. In the first subsection we characterize
graph polynomials of exponential type, describe some of  their fundamental
properties and give further examples. In the second subsection  we prove
Theorem~\ref{f-Sokal}.

In Section~\ref{2-multiplicative} we introduce the notion of
2-multiplicativity, which is a stronger version of multiplicativity. This
section is not needed for the main results, but it makes the picture more
complete.

Finally, we end the paper by some concluding remarks. 
\bigskip

\noindent \textbf{Acknowledgements.}  We are grateful to Mikl\'os Ab\'ert, 
Viktor Harangi, Tam\'as Hubai, G\'abor Tardos and the other members of
the MTA R\'enyi "Lend\"ulet" Groups and Graphs Research Group for numerous
useful discussions. The  2--connectivity statement of Theorem~\ref{2-lin} was
suggested by Tam\'as Hubai (at least for the chromatic polynomial).
\bigskip

We end this section by setting down the notations.
\medskip

\noindent \textbf{Notations.} Throughout the paper we will consider only
finite simple graphs. Connected graphs are assumed to have at least one
vertex. 2-connected graphs are assumed to have at least two vertices. Let
$\mathcal G$, $\mathcal C$ and $\mathcal C_2$ denote  the class of graphs,
connected graphs and  2--connected graphs, respectively. For a graph $G$ let
$k(G)$ denote the number of connected components of $G$.

We will follow the usual notations: $G$ is a graph, $V(G)$ is  the set of its
vertices, $E(G)$ is the set of its edges,  $e(G)$ denotes the number of 
edges, $N(x)$ is the set of the neighbors of $x$, $|N(v_i)|=\deg (v_i)=d_i$ 
denotes the degree of the vertex $v_i$. We will also use the notation $N[v]$
for the closed neighborhood $N(v)\cup \{v\}$. 

For $S\subseteq V(G)$ the graph $G-S$ denotes the subgraph of $G$ induced by
the vertices $V(G)\setminus S$ while $G[S]$ denotes the subgraph of $G$ induced
by the vertex set $S$. In case of a graph polynomial $f$ we write $f(S,x)$
instead of $f(G[S],x)$ if the graph $G$ is clear from the context. If $G$ is
a graph with vertex set $V$ and edge set $E$ then $G(V,E')$ denotes the spanning
subgraph with vertex set $V$ and edge set $E'\subseteq E$.

If $S$ is a set then $|S|$ denotes its cardinality. The notation 
$S_1\uplus S_2=V$ means that $S_1\cap S_2=\emptyset$ and $S_1\cup S_2=V$.

\section{The Tutte polynomial}\label{Tutte}

A main example that Theorem~\ref{main} applies to is  the Tutte polynomial, 
answering a question raised by Ab\'ert and Hubai. The Tutte polynomial of a
graph $G$ is defined as follows:
$$T(G,x,y)=\sum_{A\subseteq E} (x-1)^{k(A)-k(E)}(y-1)^{k(A)+|A|-|V|},$$
where $k(A)$ denotes the number of connected components of the graph $(V,A)$.
In statistical physics one often studies the following form of the Tutte
polynomial:
$$Z_G(q,v)=\sum_{A\subseteq E}q^{k(A)}v^{|A|}.$$
The two forms are essentially equivalent:
$$T(G,x,y)=(x-1)^{-k(E)}(y-1)^{-|V|}Z_G((x-1)(y-1),y-1).$$
Both forms have several advantages. For instance, it is easy to generalize the
latter one to define the multivariate Tutte polynomial. Let us assign a
variable $v_e$ to each edge and set
$$Z_G(q,\underline{v})=\sum_{A\subseteq E}q^{k(A)}\prod_{e\in A}v_e.$$
We will work with the form $Z_G(q,v)$, because the degree of this polynomial
with respect to the variable $q$ is $|V(G)|$. So it will be a bit more
convenient to work with this definition. 

Note that the chromatic polynomial of the graph $G$ is
$$ch(G,x)=Z_G(x,-1)=(-1)^{|V|-k(G)}x^{k(G)}T(G,1-x,0).$$

\begin{Pro}\label{bounded}
\begin{itemize}
\item[(a)] The chromatic polynomial is multiplicative and of  bounded
  exponential type with $R(\Delta)=e\Delta$. 
\item[(b)] For any $v_0\in\C$, the Tutte polynomial $Z_G(q,v_0)$ is
  multiplicative and of bounded exponential type with
  $R(\Delta)=e\Delta(1+|v_0|)^\Delta$ .  
\end{itemize}
\end{Pro}

\begin{proof} Multiplicativity and exponential type are evident
from the  definition. The boundedness of these polynomials can
be found in the papers \cite{sok,jps} in a stronger form.
\end{proof}

\begin{Cor} 
Theorem~\ref{main} applies to the Tutte polynomial 
$f(G,q)=Z_G(q,v_0)$ for any fixed $v_0\in\C$ and, in particular, to the
chromatic  polynomial.
\end{Cor}

This recovers the main result of \cite{abhu}, and answers  a question raised
there.

\section{Preliminaries on subgraph counting}\label{prelim}

We write $H(G)$, resp.\ $H^*(G)$  for the number of subgraphs, resp.\ induced 
subgraphs of $G$ isomorphic to $H$. Thus, by a slight abuse of notation, $
H$ denotes not only a graph but also an integer-valued, 
isomorphism-invariant function on graphs.

It is well known that if $H$ runs over the isomorhism classes of graphs, then
the  functions $H:G\mapsto H(G)$, resp.\ $H^*:G\mapsto H^*(G)$ are linearly
independent. See, for example, Lemma 4.1 in \cite{bckl}.

 Any family $\mathcal H$ of isomorphism types
of graphs gives rise to the $\C$--vector space $\C \mathcal H$, resp.\ $\C
\mathcal H^*$ generated by the functions $H$, resp.\ $H^*$, where
$H\in\mathcal H$:
$$\C \mathcal H=\left\{\sum_{H\in \mathcal{H}}c_HH(\cdot)\ |\ c_H\in \C\right\},$$
where all sums are finite.

\begin{Def} The family $\mathcal H$ is \textit{ascending} if it is closed
  under the operation of adding new edges (but no new vertices) to a graph.
\end{Def}

We need the following well-known facts.

\begin{Fact}\label{sub-ind}
If $\mathcal H$ is ascending, then $\C \mathcal H=\C \mathcal H^*$.
\end{Fact}

\begin{Fact}\label{ring}
$$H_1(G)\cdot H_2(G)=\sum_H c_{H_1,H_2}^HH(G),$$
where the coefficient $c_{H_1,H_2}^H$ is the number of decompositions of $H$
as a (not necessarily disjoint) union of $H_1$ and $H_2$. This number is a
non-negative integer and  it is zero for all but finitely many isomorphism
types $H$. 
\end{Fact}

\begin{Cor}\label{ringcoroll}
If $\mathcal H$ is closed under unions, then $\C \mathcal H$ is a ring.
\end{Cor}

\section{Multiplicative lemma}\label{multiplicative}

In this section it is shown how multiplicativity implies Benjamini--Schramm 
continuity. The main result will be Theorem~\ref{f-conv-entr}. 
Its proof is made simple by the ideas included in Lemmas~\ref{add} and 
\ref{mult}.

\begin{Def} A function $p$ defined on graphs is \textit{additive} if 
$$p(G_1 \uplus G_2)=p(G_1)+p(G_2)$$
holds for any two graphs $G_1$ and $G_2$. 
\end{Def}

\begin{Lemma} \label{add} (Additive lemma.)  A function  $p\in \C\mathcal G$ 
is additive if and only if  $p\in \C\mathcal C$.
\end{Lemma}

\begin{proof} 
The `if' part follows from the evident fact that $H(G)$ is additive (in $G$)
for any connected $H$.

We prove the `only if'  part. Let 
$$p(G)=\sum_Hc_H\cdot H(G)$$
for every graph $G$. We know that $p$ is additive and 
we need to prove that $c_H=0$ for all non-connected graphs $H$. Using the `if'
part, we may assume that  $c_H=0$ for all connected graphs $H$.

Let  $H$ be a non-connected graph. Then $H$ is the
disjoint union of  graphs $H_1$ and $H_2$ with at least 1 vertex. 
By induction, we may assume that $c_{H'}=0$ for all proper subgraphs $H'$ 
of $H$.  Then 
$$p(H_i)=\sum_{H'}c_{H'}\cdot H'(H_i)=\sum 0=0$$
for $i=1,2$ and 
$$p(H)=\sum_{H'}c_{H'}\cdot H'(H)=c_{H}\cdot H(H)=c_H.$$
By additivity we get $c_H=0$.
\end{proof}

\begin{Lemma} \label{mult} (Multiplicative lemma.) Assume that $f$ is a
multiplicative graph polynomial such that $f(G,x)$ is not the zero polynomial
for any graph $G$.    Let $\lambda_1(G),\dots ,\lambda_{n}(G)$ denote the
roots of the polynomial $f(G,x)$, where $n=\deg f(G,x)$. Assume that for some
$k$ the $k$--th power sum $p_k$ of the roots is in  $\C\mathcal G$, i.e.,
there exist constants $c_k(H)$, nonzero for only finitely many isomorphism
types $H$, such that 
$$p_k(G)=\sum_{i=1}^{n}\lambda_i(G)^k=\sum_Hc_k(H) \cdot H(G)$$
for every graph $G$. Then $p_k\in\C\mathcal C$, i.e., $c_k(H)=0$ for
non-connected graphs $H$. 
\end{Lemma}

\begin{proof} The function  $G\mapsto p_k(G)$ is additive, so the statements
  follow from Lemma~\ref{add}. 
\end{proof}

The main result of this section, Theorem~\ref{f-conv-entr}, will be an easy 
corollary of Lemma~\ref{mult}.

We recall  the following well-known fact concerning  Benjamini-Schramm
convergence.

\begin{Fact} 
Let $(G_n)$ be a graph sequence of bounded degree. Then $(G_n)$ is
Benjamini-Schramm convergent if and only if for every (finite, non-empty)
connected graph  $H$, the sequence
$$\frac{H(G_n)}{|V(G_n)|}$$
converges.
\end{Fact}

\begin{Cor} \label{conv-hom} Let $(G_n)$ be a graph sequence of bounded
degree. Then $(G_n)$ is Benjamini-Schramm convergent if and only if for
every $p\in\C\mathcal C$, the sequence
$$\frac{p(G_n)}{|V(G_n)|}$$
converges.
\end{Cor}

\begin{Th} \label{f-conv-entr}  Let $f$ be  a multiplicative monic
graph polynomial with bounded roots.   We also assume that
$$f(G,x)=\sum_{k=0}^n(-1)^{k}e_k(G)x^{n-k},$$
where $n=|V(G)|$ and all coefficients $e_k\in\C\mathcal G$.
Let $(G_n)$ be a Benjamini-Schramm convergent graph sequence.
Let $K\subset\mathbb C$ be a compact set  containing all roots of $f(G_n,x)$
for all $n$, such that $\mathbb{C}\setminus K$ is connected.
Then the statements (a) and (b) of Theorem~\ref{main} hold.
\end{Th}

\begin{Rem} Our main result, Theorem~\ref{main} follows from 
Theorem~\ref{f-conv-entr}, using also 
%Theorem~\ref{f-Sokal} and using 
the fact
that $e_k\in \C\mathcal G$ if $f$ is isomorphism-invariant and  of exponential
type  ---  this will be proved  in Theorem~\ref{lin}(a), cf.\ also
Remark~\ref{weaklin}.
Note, however, that $e_k\in\C\mathcal G$ is trivially satisfied 
for the restricted Tutte polynomial $Z_G(q,v_0)$ and hence 
also for the chromatic polynomial.

In Theorem~\ref{f-conv-entr}, we could have required that the power sums  
$p_k$ of the roots rather than the coefficients $e_k$ be in $\C\mathcal G$.
These conditions are equivalent by the Newton--Girard--Waring formulas. In
most cases, it is easier to check the condition for the coefficients, but not
always: consider the characteristic polynomial of the adjacency matrix of the
graph $G$.    
\end{Rem}

\begin{proof}[Proof of Theorem~\ref{f-conv-entr}]
This is a suitably rephrased version of  the corresponding proof of Ab\'ert
and
Hubai.

(a) We know from elementary algebra that each $p_k$ can be expressed  as a
polynomial in the $e_i$, and thus  in the functions $ H(\cdot)$. By
Corollary~\ref{ringcoroll},  $\C\mathcal G$ is a ring, so
this polynomial can be rewritten as a finite linear combination of the
functions $H(\cdot)$. From Theorem~\ref{mult} we know that this finite
linear combination consists  of terms where the graphs $H$ are  connected.
Hence $p_k\in\C\mathcal C$.
\medskip

Let $g(z)$ be continuous on $K$ and  harmonic on the interior. We need to
prove that
$$\int_Kg(z)d\mu_{G_n}(z)$$
is  convergent.
Choose  any $\varepsilon>0$.
There exists a polynomial $$h(z)=\sum_{k=0}^Ma_kz^k$$ such that
$$\left|g(z)-\Re h(z)\right|\leq \varepsilon $$
for all $z\in  K$, see for example \cite{ca}, Lemma 3. Thus,
$$\left|\int_Kg(z)d\mu_{G}(z)-\int_K\Re h(z)d\mu_{G}(z)\right|\leq
\varepsilon $$
for all graphs $G$.
Now we have
\begin{equation*}
\int_Kh(z)d\mu_G(z)=\sum_{k=0}^{M}a_k\int_Kz^kd\mu_G(z)=\sum_{k=0}^{
M}a_k\frac{p_k(G)}{|V(G)|}.
\end{equation*}
Since $(G_n)$ was Benjamini-Schramm convergent we have that
$$\frac{p_k(G_n)}{|V(G_n)|}$$
is convergent for any fixed $k$, and therefore so is
$$\int_Kh(z)d\mu_{G_{n}}(z).$$
Hence \[(\limsup-\liminf)\int_Kg(z)d\mu_{G_n}(z)\le 2\varepsilon.\]
This holds for all $\varepsilon>0$, so the integral converges.
This completes the proof of the first statement.

For the parametric version, fix $\varepsilon$ and note that any
$\xi\in \Omega$ has a neighborhood for the points of which a common
polynomial
$h(z)$ can be used in  the above argument. Therefore, the
convergence of the integral (as $n\to \infty$) is locally uniform and hence
the limit function is harmonic in
$\xi$.

(b)
We need to prove locally uniform convergence of  $ t_{n}(\xi)$. Put $\Omega=
\mathbb{C}\setminus K$, and let
$g(z,\xi)=\log |\xi-z|$ on $K\times\Omega$.

By (a), we have that
$$\int_Kg(z,\xi)d\mu_{G_n}(z)$$
converges locally uniformly in $\xi\in\Omega$. Thus
$$ t_{n}(\xi)=\frac{\log | f(G_n,\xi)|}{|V(G_n)|}=
\frac{\displaystyle{\sum_{\lambda\ \ \mbox{root}} \log |\xi-\lambda|}}{|V(G
_n)|}=
$$
$$=\int_K \log |\xi-z|d\mu_{G_n}(z)=\int_K g (z,\xi)d\mu_{G_n}(z)$$
is locally uniformly convergent as a function of $\xi$. Since
$ t_{n}(\xi)$ is  a harmonic function by its very definition, the
harmonicity of
$\lim t_{n}(\xi)$
follows from  locally uniform convergence.
\end{proof}

\section{Graph polynomials of exponential type}\label{exp}

In this section we study  graph polynomials of exponential type. 
We divided this section into two subsections. In Subsection~\ref{fund} we 
describe some fundamental properties of these polynomials. In
Subsection~\ref{gen-Sokal} we prove the generalization of
Sokal's theorem.

\subsection{Fundamental properties of  graph polynomials of exponential type
}\label{fund}

Recall that the graph polynomial $f$ is \textit{of exponential type} if 
$f(\emptyset,x)=1$  and for every graph $G=(V(G),E(G))$ we have

\begin{equation}\label{exp-type}
\sum_{S_1\uplus S_2=V(G)}f(S_1,x)f(S_2,y)=f(G,x+y),
\end{equation} 
where $f(S_1,x)=f(G[S_1],x), f(S_2,y)=f(G[S_2],y)$ are the polynomials of
 the subgraphs of $G$ induced by the sets $S_1$ and $S_2$, respectively.
\medskip

Note that Gus Wiseman \cite{wis} calls these graph polynomials
\textit{binomial-type}. 

We will see that the chromatic polynomial is a graph polynomial of exponential
type.  In fact, this class contains surprisingly many known graph
polynomials including the Laplacian characteristic polynomial and a modified
version of the matching polynomial.

Our first aim is to describe  graph polynomials of exponential type. It will
turn out that if we have a function $b$ from the class of  
graphs with at least one vertex to the complex numbers then this function 
determines a graph polynomial of exponential type and vice-versa every  graph
polynomial of exponential type  determines such a function. 
The next theorem makes this statement more precise.

\begin{Th} \label{exp-bb} Let $b$ be a complex-valued function on the class 
of  graphs on non-empty vertex sets. Let us define the graph polynomial $f_b$ 
as follows. Set
$$a_k(G)=\sum_{\{S_1,S_2,\dots ,S_k\}\in \mathcal{P}_k}b(S_1)b(S_2)\dots b(S_k),$$
where the summation is over the set $\mathcal{P}_k$ of all  partitions of
$V(G)$ into exactly $k$ non-empty sets. Then let 
$$f_b(G,x)=\sum_{k=1}^na_k(G)x^k,$$
where $n=|V(G)|$.
Then
\begin{itemize}
\item[(a)] For any function $b$, the graph polynomial $f_b(G,x)$ is of
  exponential type. 

\item[(b)]  For any  graph polynomial $f$ of exponential type, there exists 
a graph function $b$ such that $f(G,x)=f_b(G,x)$. More precisely, if
$b(G)=a_1(G)$ is the coefficient of $x^1$ in $f(G,x)$, then $f=f_b$. 
\end{itemize}
\end{Th}

\begin{proof}
\noindent (a) This is a formal computation:
$$\sum_{S_1\uplus S_2=V(G)}f_b(S_1,x)f_b(S_2,y)=
\sum_{S_1\uplus S_2=V(G)}\left(\sum_{k=1}^na_k(S_1)x^k\right)
\left(\sum_{\ell=1}^na_{\ell}(S_2)y^{\ell}\right)=$$
$$=\sum_{S_1\uplus S_2=V(G)}\left(\sum_{k=1}^n\left(
\sum_{\{R_1,R_2,\dots ,R_k\}\in \mathcal{P}_k(S_1)}b(R_1)b(R_2)\dots
b(R_k)\right)x^k\right)\cdot$$  
$$\cdot \left(\sum_{\ell=1}^n\left(\sum_{\{T_1,T_2,\dots ,T_{\ell}\}\in
\mathcal{P}_{\ell}(S_2)}b(T_1)b(T_2)\dots b(T_{\ell})\right)y^{\ell}\right)=$$ 
$$=\sum_{r=1}^n\left(\sum_{\{Q_1,Q_2,\dots ,Q_r\}\in
\mathcal{P}_r(G)}b(Q_1)\dots b(Q_r)\right)\left( \sum_{i=1}^r{r \choose
  i}x^iy^{r-i}\right)=$$ 
$$=\sum_{r=1}^na_r(G)(x+y)^r=f_b(G,x+y).$$
Hence we have proved part (a). 
\bigskip

\noindent (b) The following proof is due to G\'abor Tardos \cite{tar}.
We prove the statement by induction on the number of vertices. The claim is
trivial  for the empty graph. Observe that by induction
$$f(G,x+y)-f(G,x)-f(G,y)=\sum_{S_1\uplus S_2=V(G)\atop  S_1,S_2\neq
  \emptyset}f(S_1,x)f(S_2,y)=$$ 
$$=\sum_{S_1\uplus S_2=V(G)\atop  S_1,S_2\neq
  \emptyset}f_b(S_1,x)f_b(S_2,y)=f_b(G,x+y)-f_b(G,x)-f_b(G,y).$$ 
In the last step we used the fact that $f_b$ is exponential-type by part (a).
Hence for the polynomial $g(x)=f(G,x)-f_b(G,x)$ we have
$$g(x+y)=g(x)+g(y).$$
Thus $g(x)$ is linear: $g(x)=cx$. On the other hand, $b(G)=a_1(G)$ is 
defined as the coefficient of $x^1$ in $f(G,x)$, whence $c=0$.
This completes the proof.
\end{proof}

\begin{Rem} In the ``nice cases'' we have $f(K_1,x)=x$ or in other words,
$b(K_1)=1$ implying that $f$ is  monic, but this is not necessarily true in
general.  
\end{Rem}

\begin{Th}\label{mult-conn} Let $f_b(G,x)$ be a graph polynomial of
exponential type.  Then $f_b$ is multiplicative if and only if $b$ vanishes on
non-connected graphs. 
\end{Th}

\begin{proof} Since the constant term of an exponential type polynomial is
$0$ for every  graph with at least one vertex, the condition is necessary:
if $H=H_1\uplus H_2$ then $f_b(H)=f_b(H_1)f_b(H_2)$ implies that $b(H)=0$.

On the other hand, if $b(H)=0$ for all non-connected graphs then from 
Theorem~\ref{exp-bb} we see that
$$a_k(H_1\uplus H_2)=\sum_{j=1}^ka_j(H_1)a_{k-j}(H_2)$$
which means that $f(H_1\uplus H_2,x)=f(H_1,x)f(H_2,x)$.
\end{proof}

\begin{Rem}\label{ex} Surprisingly, the class of multiplicative  graph 
polynomials of bounded exponential type contains other natural graph
polynomials besides the chromatic polynomial. 
We give some examples of these polynomials with their function $b$ without
proof. 
\bigskip

\noindent $\bullet$ Let
$M(G,x)=x^n-m_1(G)x^{n-1}+m_2(G)x^{n-2}-m_3(G)x^{n-3}+\dots $ be the
(modified) matching polynomial \cite{god3,god4,hei}, where $m_k(G)$ is the
number of matchings of size $k$. Then $M(G,x)$ is of exponential type, where
the function $b_M$ satisfies $b_M(K_1)=1$, $b_M(K_2)=-1$ and $b_M(H)=0$
otherwise. 
\medskip

\noindent $\bullet$ Let $h(G,x)=\sum_{k=1}^n(-1)^{n-k}a_k(G)x^k$ be the
adjoint polynomial, where $a_k(G)$ is the number of ways one can cover the
vertex set of $G$ by $k$ vertex disjoint complete graphs. Then $h(G,x)$ is of
exponential type, where the function $b_h$ satisfies $b_h(K_n)=(-1)^{n-1}$
for the complete graphs $K_n$ and $b_h(H)=0$ otherwise. 
\medskip

\noindent $\bullet$ Let $L(G,x)$ be the characteristic polynomial of the
Laplacian matrix, for the sake of brevity we will call it the Laplacian
polynomial. The Laplacian matrix $L(G)$ is defined as follows:
$L(G)_{ii}=d(i)$ is the degree of the vertex $i$ and if $i\neq j$ then
$L(G)_{ij}$ is $(-1)$ times the number of edges connecting the vertices $i$
and $j$. Then $L(G,x)$ is of exponential type, where the function $b_L$ 
satisfies $b_L(G)=(-1)^{|V(G)|-1}|V(G)|\tau(G)$, where $\tau(G)$ is the number 
of spanning trees of the graph $G$. 
\medskip

\noindent $\bullet$ The polynomial $Z_G(q,v_0)$ is of exponential type for 
every fixed $v_0$. In fact, one can easily  prove the following identity for
the multivariate version  of the Tutte polynomial \cite{sco-sok2}:
$$\sum_{S\subseteq V(G)}Z_{G[S]}(q_1,\underline{v})Z_{G[V\setminus
    S]}(q_2,\underline{v})=Z_G(q_1+q_2,\underline{v}).$$  
\end{Rem}

\begin{Rem}\label{weaklin} To prove Theorem~\ref{main}, we only need to be 
able to apply  Theorem~\ref{f-conv-entr}. For that, we  need to prove that 
for $f$ isomorphism-invariant and of exponential type, 
$e_k\in\C\mathcal G$ holds. However, we include the following more precise 
statements for the sake of completeness. Recall that if $\mathcal{H}$ is a
class of graphs then $\C \mathcal{H}$ is the vector space generated by the 
functions $H(\cdot )$ $(H\in \mathcal{H})$. The following theorem asserts that
for exponential-type graph polynomials, the coefficients $e_k$ and power sums
$p_k$ are in $\C \mathcal{H}_k$, where $\mathcal{H}_k$ is a
very special class of graphs.  
\end{Rem}

\begin{Th}\label{lin} Let $$f(G,x)=\sum_{k=0}^n(-1)^{k}e_k(G)x^{n-k}$$ be an
isomorphism-invariant  monic graph polynomial of exponential type, where
$n=|V(G)|$. Let  $k\ge 1$. Define $p_k(G)$ to be the $k$-th power sum of the
roots of $ f(G,x)$.

\begin{itemize}
\item[(a)]  We have 
\begin{equation}\label{e}
e_k\in\C \{H : k+1\le |V(H)|\le 2k, |V(H)|-k(H)\le k\}
\end{equation} 
and
\begin{equation}\label{p}
p_k\in\C \{H : 2\le |V(H)|\le 2k, |V(H)|-k(H)\le k\}.
\end{equation}
\item[(b)] If, in addition, $f$  is multiplicative, then
\begin{equation}\label{e-mult}
e_k\in\C \{H : H \textit{ has no isolated points and } |V(H)|-k(H)= k\}
\end{equation} 
and
\begin{equation}\label{p-mult}p_k\in \C \{H :  H \textit{ is connected and 
} 2\le |V(H)|\le k+1\}.
\end{equation}
\end{itemize}
\end{Th}

\begin{proof}
We will apply Theorem~\ref{exp-bb}. 
Note that
\[e_k(G)=(-1)^{k}a_{n-k}(G).\]
By  Theorem~\ref{exp-bb} we have
$$a_{n-k}(G)=\sum_{\{S_1,S_2,\dots ,S_{n-k}\}\in
  \mathcal{P}_{n-k}}b(S_1)\cdots b(S_{n-k}),$$ 
where the summation runs over the set $\mathcal{P}_{n-k}$ of all  partitions
of $V(G)$ into exactly $n-k$ non-empty sets. 
Let us consider the partition $\pi=\{S_1,\dots ,S_{n-k}\}$. Let it  have
$n-l$ singletons and $l-k$ bigger parts. Then the union $R$ of these bigger
parts has $l$ points and $k+1\le l\le 2k$. We group the terms of the last sum
according to the isomorphism class of $G[R]$:

\begin{equation}\label{induced}
a_{n-k}(G)=\sum_{k+1\le |V( H)|\le 2k}H^*(G)\sum_{\bar\pi=\{S_1,S_2,\dots
  ,S_{|H|-k}\}}b(S_1)\cdots b(S_{|H|-k}), 
\end{equation}
where  $\bar\pi$ runs over the partitions of $V(H)$ into  $|H|-k$ sets of at
least $2$ elements. By Fact~\ref{sub-ind} this can be rewritten as

\begin{equation}\label{coeff}
a_{n-k}(G)=\sum_{k+1\le |V( H)|\le 2k}c_k(H)\cdot H(G),
\end{equation}
where $c_k(H)\in\C$. This already proves that $e_k\in\C\mathcal G$.
Note that  formula~\eqref{coeff} holds for $k=0$ as well, except that we no
longer have $k+1\le |V(H)|$. More precisely, $c_0(\emptyset)=1$ and $c_0(H)=0$
for all non-empty graphs $H$. 

We need to prove that $c_k(H)=0$ unless $|V(H)|-k(H)\le k$. 
The defining identity~\eqref{exp-type} of exponential type and 
formula~\eqref{coeff} yields
$$\sum_k\sum_H c_k(H)\cdot H(G) (x+y)^{n-k}=$$ 
$$=\sum_{S\subseteq V(G)}\sum_i\sum
_{H'} c_i(H')\cdot H'(S)x^{|S|-i}\sum_j\sum _{H''} c_j(H'')\cdot
H''(G-S)y^{n-|S|-j}=$$
$$=\sum_{H'}\sum _{H''} \sum_i\sum_jc_i(H')c_j(H'')\cdot (H'\uplus H'')(G)
(x+y)^{n-|H'|-|H''|}x^{|H'|-i}y^{|H''|-j}. 
$$
We consider the part that is homogeneous of degree $n-k$ in $x$ and $y$:
$$\sum_H c_k(H)\cdot H(G) (x+y)^{n-k}=$$ 
$$=\sum_{H'}\sum _{H''} \sum_{i+j=k}c_i(H')c_j(H'')\cdot (H'\uplus H'')(G)
(x+y)^{n-|H'|-|H''|} x^{|H'|-i}y^{|H''|-j}.$$
This is true for all $G$, therefore
$$c_k(H) (x+y)^{n-k}=\sum_{H'\uplus H''=H} \sum_{i+j=k}c_i(H')c_j(H'') 
(x+y)^{n-|H|}x^{|H'|-i}y^{|H''|-j}.$$
We divide by $(x+y)^{n-|H|}$ to get
$$c_k(H)(x+y)^{|H|-k}=\sum_{H'\uplus H''=H}\sum_{i+j=k}c_i(H')c_j(H'') 
x^{|H'|-i}y^{|H''|-j}=$$
$$=c_k(H)(x^{|H|-k}+y^{|H|-k})+\sum_{H'\uplus H''=H}
\sum_{i=1}^{k-1}c_i(H')c_{k-i}(H'') x^{|H'|-i}y^{|H''|-k+i}.$$
If  $|V(H)|= k+1,$ then  $|V(H)|-k(H)\le k$ trivially holds and there is 
nothing to prove. If  $|V(H)|\ge k+2$ and $c_k(H)\ne 0$, then 
$$0\ne  c_k(H) \left((x+y)^{|H|-k}-x^{|H|-k}-y^{|H|-k}\right)=$$
$$=\sum_{H'\uplus H''=H}  \sum_{i=1}^{k-1}c_i(H')c_{k-i}(H'')
x^{|H'|-i}y^{|H''|-k+i}.$$
Thus there exists $1\le i\le k-1$ and a decomposition $H'\uplus H''=H$ with 
$c_i(H')c_{k-i}(H'')\ne 0$. By induction on $k$, we may assume that 
$|V(H')|-k(H')\le i$ and $|V(H'')|-k(H'')\le k-i$, whence $|V(H)|-k(H)\le k$.
This proves the  statement~\eqref{e} of  part (a).

For the  statement~\eqref{p}, recall from elementary algebra that $p_k$ is an 
integral linear combination of terms $\prod e_{i_j}$, where $i_j\ge 1$ and
$\sum i_j=k$. The statement~\eqref{p} thus follows from statement~\eqref{e} 
and Fact~\ref{ring}.

In part (b), we assume that $f$ is multiplicative. By Theorem~\ref{mult-conn}, 
$b$ vanishes on non-connected graphs. Thus, from equation~\eqref{induced} and
Fact~\ref{sub-ind} we see that 
$$e_k\in\C \{H : H \textit{ has no isolated points and }  |V(H)|-k(H)\ge
k\}.$$  
Using \eqref{e}, the statement \eqref{e-mult} follows. The 
statement~\eqref{p-mult} is immediate from Lemma~\ref{mult} and the
statement~\eqref{p}. 
\end{proof}

\subsection{Proof of  Theorem~\ref{f-Sokal} (the Sokal bound)}\label{gen-Sokal}

To prove Theorem~\ref{f-Sokal}, we will need the concept of the multivariate
independence polynomial.

\begin{Def} Let $w: V(G)\rightarrow \mathbb{C}$ and let
$$I(G,\underline{w})=\sum_{S\in \mathcal{I}} \prod_{u\in S}w_u $$
be the \textit{multivariate independence polynomial}.
\end{Def}

Our strategy will be the following. We express any exponential-type graph
polynomial as an independence polynomial of a bigger graph. Then we use 
Dobrushin's lemma and its corollaries to deduce Theorem~\ref{f-Sokal}.

\begin{Th}[Dobrushin's Lemma] Assume that the functions $r:V(G)\rightarrow (0,1)$ and
  $w: V(G)\rightarrow \mathbb{C}$ satisfy the inequalities 
$$|w_v|\leq (1-r(v))\prod_{(v,v')\in E(G)}r(v')$$
for every $v\in V(G)$. Then
\begin{itemize}
\item[(a)]
$I(A,\underline{w})\neq 0$ for every $A\subseteq V(G)$.
\item[(b)]
$$\left|\frac{I(B,\underline{w})}{I(A,\underline{w})}\right|\leq 
\left(\prod_{u\in A\setminus B}r(u)\right)^{-1}$$ 
for every $B\subseteq A\subseteq V(G)$.
\end{itemize}
\end{Th}

\begin{proof} (The following proof is a very slight modification of Borgs'
  proof \cite{bor}.) 
We prove the two statements together by induction on the size of $|A|$. Both
statements are trivial if $|A|=0$. Now assume that we have already proved 
both statements for sets of size smaller than $|A|$. Now let us prove them
for the set $A$. Let $v\in A$.  Recall that $N[v]=N(v)\cup \{v\}$.
$$I(A,\underline{w})=I(A-v,\underline{w})+w_vI(A-N[v],\underline{w})=$$
$$=I(A-v,\underline{w})\left(1+w_v\frac{I(A-N[v],\underline{w})}{I(A-v,
\underline{w})}\right).$$   
By induction $I(A-v,\underline{w})\neq 0$ so we can indeed divide by it. By
part (b) of the induction we have
$$\left|w_v\frac{I(A-N[v],\underline{w})}{I(A-v,\underline{w})}\right|\leq
|w_v| \left(\prod_{v'\in (A-v)\setminus (A-N[v])}r(v')\right)^{-1}\leq $$
Now let us use the condition on $w_u$'s. (Clearly, $(A-v)\setminus
(A-N[v])=\{v'\ |\ (v,v')\in E(A)\}$.)  
$$\leq (1-r(v))\prod_{(v,v')\in E(G)}r(v') \left(\prod_{v'\in (A-v)\setminus
  (A-N[v])}r(v')\right)^{-1}=1-r(v).$$ 
Hence
$$|I(A,\underline{w})|=|I(A-v,\underline{w})|\left|1+w_v\frac{I(A-N[v],
\underline{w})}{I(A-v,\underline{w})}\right|\geq $$  
$$\geq
|I(A-v,\underline{w})|\left(1-\left|w_u\frac{I(A-N[v],\underline{w})}{I(A-v,
  \underline{w})}\right|\right)\geq $$    
$$\geq |I(A-v,\underline{w})|r(v)>0$$
Hence $I(A,\underline{w})\neq 0$. Finally, if $B\subset A$ then for any $v\in
A\setminus B$ we have $B\subseteq A-v$ so we obtain
$$\left|\frac{I(B,\underline{w})}{I(A,\underline{w})}\right|=\left|\frac{I
(B,\underline{w})}{I(A-v,\underline{w})}\right|\left|\frac{I(A-v,\underline{
w})}{I(A,\underline{w})}\right|\leq $$    
$$\leq  \left(\prod_{u\in (A-v)\setminus B}r(u)\right)^{-1}r(v)^{-1}=
\left(\prod_{u\in A\setminus B}r(u)\right)^{-1}.$$
Hence we have proved part (b) as well.
\end{proof}

\begin{Rem} The condition of  Dobrushin's lemma resembles  that of the
Lov\'asz local lemma. This is not a coincidence. Scott and Sokal
\cite{sco-sok1} gave an exact form of the Lov\'asz local lemma in terms of
the independence polynomial. Although their theorem is precise, it is hard to
use it in its original form. Dobrushin's lemma provides a way to obtain a
useful relaxation of their theorem. In this way we recover the original
Lov\'asz local lemma.  
\end{Rem} 

In our application, it will be more convenient to check the condition of the
following version of  Dobrushin's lemma.

\begin{Cor} \label{Dob-cor} If the function $a:V(G)\rightarrow \mathbb{R}^+$
satisfies the inequality 
$$\sum_{u\in N[v]}\log \left(1+|w_u|e^{a(u)}\right)\leq a(v)$$
for each vertex $v\in V(G)$, then 
$$I(G,\underline{w})\neq 0.$$
In particular, if the function $a:V(G)\rightarrow \mathbb{R}^+$ satisfies the
inequality 
$$\sum_{u\in N[v]}|w_u|e^{a(u)}\leq a(v)$$
for each vertex $v\in V(G)$, then 
$$I(G,\underline{w})\neq 0.$$
\end{Cor} 

\begin{proof} Set $r(v)=\left(1+|w_v|e^{a(v)}\right)^{-1}$.
Then the condition
$$|w_v|\leq (1-r(v))\prod_{(v,v')\in E(G)}r(v')$$
can be rewritten using
$$(1-r(v))\prod_{(v,v')\in
  E(G)}r(v')=\left(1-\frac{1}{1+|w_v|e^{a(v)}}\right)\prod_{u\in
  N(v)}\frac{1}{1+|w_u|e^{a(u)}}=$$  
$$=|w_v|e^{a(v)}\prod_{u\in N[v]}\frac{1}{1+|w_u|e^{a(u)}}.$$
Hence the condition of  Dobrushin's lemma is equivalent to
$$\sum_{u\in N[v]}\log \left(1+|w_u|e^{a(u)}\right)\leq a(v).$$
The second statement simply follows from the inequality $\log (1+s)\leq s$ 
for $s\geq 0$. 
\end{proof}

We can express a (monic) graph polynomial of exponential type as a
multivariate independence polynomial as follows.   
\bigskip

Let us define the graph $\tilde{G}=(\tilde{V},\tilde{E})$ as
follows. The vertex set $\tilde{V}$ consists of the subsets of $V(G)$ of
size at least $2$. If $S_1,S_2\subseteq V(G)$ such that $|S_1|,|S_2|\geq 2$
then let $(S_1,S_2)\in \tilde{E}$ if and only if $S_1\cap S_2\neq
\emptyset$. Hence the independent sets of $\tilde{G}$ are those subsets
$S_1,\dots ,S_k\subset V(G)$ for which $S_i\cap S_j=\emptyset$ for 
$1\leq i<j\leq k$ and $|S_i|\geq 2$. In what follows we will call
$\{S_1,S_2,\dots ,S_k\}$ a subpartition. The set of subpartitions will be
denoted by $\mathcal{P}'(G)$. 

Let us define the function (or weights) of the vertices of $\tilde{G}$ as
follows. Set
$$w(S)=\frac{b(S)}{q^{|S|-1}}.$$
Note that if we have a subpartition $\{S_1,S_2,\dots ,S_k\}$ then by dividing \\
$V(G)\setminus \bigcup_{i=1}^kS_i$ into $1$-element sets we get a partition 
of $V(G)$ into 
$$k+n-\sum_{i=1}^k|S_i|=n-\sum_{i=1}^k(|S_i|-1)$$ 
sets. Then 
$$I(\tilde{G},\underline{w})=\sum_{\{S_1,\dots ,S_k\}\in
  \mathcal{P}'(G)}\prod_{i=1}^k\frac{b(S_i)}{q^{|S_i|-1}}=q^{-n}f_b(G,q).$$  
Hence if we find a zero-free region for $I(\tilde{G},\underline{w})$ then we
find a zero-free region for $f_b(G,x)$. 

Now we will specialize Corollary~\ref{Dob-cor} to our case.

\begin{Lemma} \label{Dob-set} Assume that for some positive real number $a$ 
we have 
$$\sum_{S\in \tilde{V} \atop v\in
  S}|b(S)|\left(\frac{e^a}{|q|}\right)^{|S|-1}\leq ae^{-a}$$ 
for every $v\in V(G)$. 
Then $f_b(G,q)\neq 0$.
\end{Lemma}

\begin{proof} Let us check the second condition of Corollary~\ref{Dob-cor} 
for the weight function $w(S)=\frac{b(S)}{q^{|S|-1}}$ and the function
$a(S)=a|S|$. We need that
$$\sum_{S: (T,S)\in \tilde{E}}|w(S)|e^{a|S|}\leq a|T|$$
for every $T\in \tilde{V}$. (Recall that $(T,S)\in \tilde{E}$ means that
$S\cap T\neq \emptyset$.)
This is indeed true since 
$$\sum_{S: (T,S)\in \tilde{E}}|w(S)|e^{a|S|}\leq \sum_{v\in T}\sum_{v\in
  S}|w(S)|e^{a|S|}=$$  
$$=\sum_{v\in T}\sum_{v\in S}\frac{|b(S)|}{q^{|S|-1}}e^{a|S|}=\sum_{v\in
 T}e^a \sum_{v\in S}|b(S)|\left(\frac{e^a}{|q|}\right)^{|S|-1}\leq $$
By the assumption of the lemma, the inner sum can be bounded above by
$ae^{-a}$. The whole expression is therefore
$$\leq \sum_{v\in T}e^a(ae^{-a})=a|T|,$$
as claimed.
\end{proof}

Now we are ready to prove Theorem~\ref{f-Sokal}. 

\begin{proof} Since $f$ is of bounded exponential type, we have 
$$\sum_{S: v\in S \atop |S|=t}|b(S)|\leq R(\Delta)^{t-1},$$
whence 
$$\sum_{S\in \tilde{V} \atop v\in S}|b(S)|\left(\frac{e^a}{|q|}\right)^{|S|-1}=
\sum_{t=2}^{\infty}\left(\sum_{S: v\in S \atop |S|=t}|b(S)|\right) 
\left(\frac{e^a}{|q|}\right)^{t-1}\leq$$
$$\leq \sum_{t=2}^{\infty}R(\Delta)^{t-1}\left(\frac{e^a}{|q|}\right)^{t-1}
=\frac{1}{1-\frac{R(\Delta)e^a}{|q|}}-1$$
So the polynomial $f(G,x)$ does not vanish at  $q$ if
$$\frac{1}{1-\frac{R(\Delta)e^a}{|q|}}\leq 1+ae^{-a}.$$
This is satisfied if 
$$|q|\geq (e^a+e^{2a}/a)R(\Delta).$$
Hence the roots of the polynomial $f(G,z)$ have absolute value
$<c\cdot R(\Delta),$ where
$$c=\min_{a\geq 0}(e^a+e^{2a}/a)\approx 7.0319,$$
and the minimalizing $a\approx 0.4381$.
\end{proof}

\begin{Rem} The condition of Definition~\ref{cond-exp} seems to be
frightening, but in several cases it is easy to check. For the matching
polynomial or the adjoint polynomial  $R(\Delta)=\Delta$  trivially
satisfies the inequality. We get a slightly better bound if we apply
Lemma~\ref{Dob-set} directly. On the other hand, it can be proved that the
sharp upper bound for the roots is $4(\Delta-1)$ in both cases. For the
matching polynomial this is proved in \cite{hei}, for the
adjoint polynomial this is proved in \cite{csi_adj}.
\end{Rem}

\section{2-multiplicativity and 2-connected graphs}\label{2-multiplicative}

In a strict sense this part is not connected to the main objective of this
paper, but  it makes the picture more complete for the most interesting graph
polynomials. 

We know that the chromatic polynomial and the restricted Tutte polynomial
$Z(q,v_0)$ for any fixed $v_0$ are multiplicative. In fact, they are even
2-multiplicative in the following sense. 

\begin{Def}  We say that the graph polynomial $f$ is \textit{2-multiplicative}
  if it is multiplicative and
$$xf(G_1\cup G_2,x)=f(G_1,x)f(G_2,x),$$
where $G_1\cup G_2$ is any  union of the graphs $G_1$ and $G_2$ such that 
they have exactly one vertex in common.
\end{Def}

We have seen that the multiplicativity of a graph polynomial implies that
the coefficients of non-connected graphs vanish in the $k$-th power sum $p_k$:
see Lemma~\ref{mult}. The following statements show that 2-multiplicativity
implies that the coefficients of non-2-connected graphs vanish in the $k$-th
power sum $p_k$. Recall that $\mathcal{C}_2$ denotes the class of 2-connected
graphs. 

\begin{Lemma} \label{2-mult} (2-multiplicative lemma.) Assume that $f$ is a
2-multiplicative graph polynomial such that $f(G,x)$ is not the zero 
polynomial for any graph $G$.  For some $k\ge 1$, assume that the $k$-th power
sum of the roots $p_k\in\C\mathcal G$. Then $p_k\in\C\mathcal C_2$. 
\end{Lemma}

Naturally, the proof of Lemma~\ref{2-mult} relies on the corresponding
analogue of Lemma~\ref{add}.

\begin{Def} A function $p$ defined on graphs is \textit{2-additive} if 
\[p(G_1 \cup G_2)=p(G_1)+p(G_2)\] 
holds for the  union of any two graphs $G_1$ and $G_2$ that have at most one
vertex in common.  
\end{Def}

\begin{Lemma} \label{2-add} (2-additive lemma.)  A function  $p\in \C\mathcal G$ 
is 2-additive if and only if $p\in \C\mathcal C_2$.
\end{Lemma}

The proof of this lemma is the straightforward analogue of the proof of
Lemma~\ref{add}. 

\begin{proof}[Proof of Lemma~\ref{2-mult}.] The function 
$G\mapsto p_k(G)$ is 2-additive, so the statement follows 
from Lemma~\ref{2-add}. 
\end{proof}
 
For the sake of completeness, we include the analogues of those theorems where
multiplicativity played some role. The following theorem gives a  description
of  2-multiplicative graph polynomials of exponential type.

\begin{Th}\label{2-mult-conn} Let $f_b(G,x)$ be a graph polynomial of 
exponential type. Then $f_b$ is 2-multiplicative if and only if 
\begin{itemize}
\item[(i)] $b$ vanishes on  non-connected graphs and 
\item[(ii)] $b(H_1\cup H_2)=b(H_1)b(H_2)$ whenever $H_1$ and $H_2$ have
  exactly one vertex in common.  
\end{itemize}
\end{Th}

\begin{proof} Since 2-multiplicativity implies multiplicativity, the
  condition (i) is necessary.   For (ii), note that  if $H=H_1\cup H_2$ and 
$|V(H_1)\cap V(H_2)|=1$ then by taking the coefficient of $x^2$ in
$xf_b(H)=f_b(H_1)f_b(H_2)$ we get the statement.

On the other hand, if (i) and (ii) hold, then $f_b$ is multiplicative and from
Theorem~\ref{exp-bb}  we see that  if $|V(H_1)\cap V(H_2)|=1$ then
$$a_k(H_1\cup H_2)=\sum_{j=1}^ka_j(H_1)a_{k+1-j}(H_2)$$
which means that $xf(H_1\cup H_2,x)=f(H_1,x)f(H_2,x)$.
\end{proof}

The following theorem is the natural ``third part'' of Theorem~\ref{lin}. 

\begin{Th}\label{2-lin} Let $$f(G,x)=\sum_{k=0}^n(-1)^{k}e_k(G)x^{n-k}$$ be an
isomorphism-invariant  monic graph polynomial of exponential type, where
$n=|V(G)|$. Let  $k\ge 1$. Define $p_k(G)$ to be the $k$-th power sum of the
roots of $ f(G,x)$. If $f$  is 2-multiplicative, then
\begin{equation*}\label{p-2-mult}
p_k\in \C \{H :  H \textit{ is 2-connected and } 2\le |V(H)|\le k+1\}.
\end{equation*}
\end{Th}

\begin{proof}
The statement follows from  Lemma~\ref{2-mult} and \eqref{p-mult}. 
\end{proof}

\begin{Rem} \label{coeff-formula} It is also interesting to note that if we
  consider the chromatic polynomial, then in the expression
$$p_k(\cdot)=\sum_{H}c_{k}(H) H(\cdot),$$
all coefficients $c_{k}(H)$ are integers. In fact, this is always true if the
parameters $b(\cdot)$ defining the exponential type graph polynomial are
integers. So it is also true for the matching, adjoint, Laplacian polynomials.
This can be seen from the proof of Theorem~\ref{lin}. 

Moreover, by looking at  that proof more closely, it is also easy to obtain
an explicit formula for $c_k(H)$.  Let $f$ be an isomorphism-invariant,
multiplicative monic graph polynomial of exponential type. By
Theorem~\ref{lin}(b), we have
\[(-1)^ie_i(G)=\sum_Hc(H)H(G),\] where $H$ runs over graphs without isolates
satisfying $|V(H)|-k(H)=i$.  Then, for $k\ge 1$,
\begin{equation}\label{formula}
c_k(H) =  k  \sum_{H_1,...,H_m}  (-1)^m\frac{(m-1)!} { \prod_i m_i!}
\prod_{j=1}^m c(H_j).
\end{equation}
The summation is over sequences
$ H_1$, \dots, $H_m $ of isolate-free subgraphs of $H$    which satisfy $
V(H_1)\cup \dots \cup V(H_m)=V(H) $,  $ E(H_1)\cup \dots \cup E(H_m)=E(H) $
and
$\sum im_i  =  k$, where $m_i$ is the number of indices $j$ such that
$|V(H_j)|-k(H_j)=i$.

Assume further that $f$ is 2-multiplicative.
By Theorem~\ref{2-lin}, the sum in \eqref{formula} is zero unless $H$ is
2--connected and $2\le |V(H)|\le k+1$.

When $f(G,q)$ is the Tutte polynomial $ Z_G(q,v)$,  we have
$c(H)=v^{|E(H)|}$, so

\begin{equation}\label{subcoeff}
c_k(H) =  k  \sum_{H_1,...,H_m}  (-1)^m v^{\sum |E(H_j)|}\frac{(m-1)!} {
\prod m_i!}.
\end{equation}

For the chromatic polynomial, we substitute $v=-1$ in \eqref{subcoeff}.
\end{Rem}

\section{Concluding remarks}

Clearly, the conditions of Theorem~\ref{f-conv-entr} are very weak and most of
the well-known graph polynomials satisfy them. The class of multiplicative
graph polynomials of bounded exponential type  already contains the matching
polynomial, the chromatic polynomial (in fact, all restricted Tutte
polynomials) and the Laplacian characteristic polynomial. On the other hand,
the characteristic polynomial of the adjacency matrix is not of exponential
type, but trivially satisfies all conditions: it is multiplicative, all its
roots have absolute value at most $\Delta$ and $p_k(G)=\hom(C_k,G)$ (the
number of closed walks of length $k$). 

It is  not known how the condition $\xi\not\in K$ in
Theorem~\ref{f-conv-entr}(b) can be relaxed. R.\ Lyons \cite{lyons} proved that
if $(G_n)$ is Benjamini-Schramm convergent, then 
$(\log \tau(G_n))/{|V(G_n)|}$
is convergent, where $\tau(G_n)$ denotes the number of spanning trees of the
graph $G_n$. This result can be expressed as a result concerning a special
value of the Tutte polynomial, namely $\tau(G)=T(G,1,1)$. It is clearly not a
special case of our theorems, because this point is deep inside the disc $K$
given by the Sokal bound.


\begin{thebibliography}{99} 





\bibitem{abhu} M. Ab\'ert and T. Hubai: \textit{Benjamini-Schramm convergence
  and the distribution of chromatic roots for sparse graphs}, to appear in Combinatorica,  arXiv preprint
 1201.3861


\bibitem{bor} C. Borgs: \textit{Absence of Zeros for the Chromatic
Polynomial on Bounded Degree Graphs}, Combinatorics, Probability and
Computing \textbf{15} (2006) No. 1-2, pp. 63-74  

\bibitem{bckl} C. Borgs, J.\ Chayes, J.\ Kahn and L.\ Lov\'asz: 
\textit{Left and right convergence of graphs with bounded degree}, Random Structures and Algorithms \bf 42 \rm (2013), no. 1, 1--28. 

\bibitem{ca} L. Carleson: \textit{Mergelyan's theorem on uniform polynomial
approximation}, Math. Scand. \textbf{15} (1964), pp.\ 167-175


\bibitem{csi_adj} P. Csikv\'ari: \textit{Two remarks on the adjoint
polynomial}, European Journal of Combinatorics \textbf{33} (2012), pp.\ 
583-591 

\bibitem{god3} C. D. Godsil: \textit{Algebraic Combinatorics}, Chapman and
Hall, New York 1993 

\bibitem{god4} C. D. Godsil and I. Gutman: \textit{On the theory of the
  matching polynomial},  J. Graph Theory \textbf{5} (2006), pp. 137-144,  


\bibitem{hei}  O. J. Heilmann and E. H. Lieb: \textit{Theory of monomer-dimer 
systems}, Commun. Math. Physics \textbf{25} (1972), pp. 190-232


\bibitem{jps} B. Jackson, A. Procacci, A. D. Sokal: \textit{Complex
  zero-free regions at large $|q|$ for multivariate Tutte polynomials (alias
  Potts-model partition functions) with general complex edge weights}, J.\ Combin. Theory Ser.\ B \bf 103 \rm (2013), no. 1, 21--45.
  

\bibitem{lyons} R. Lyons: \textit{Asymptotic enumeration of spanning trees},
  Combinatorics, Probability and Computing \textbf{14} (2005), pp. 491-522

\bibitem{rea} R. C. Read: \textit{An introduction to Chromatic Polynomials},
Journal of Combinatorial Theory \textbf{4} (1968), pp. 52-71


\bibitem{sco-sok1} A. D. Scott and A. D. Sokal: \textit{The repulsive
lattice gas, the independent-set polynomial, and the Lov\'asz local lemma}, 
J. Stat. Phys. \textbf{118} (2005), No. 5-6, pp.  1151-1261. 



\bibitem{sco-sok2} A. D. Scott and A. D. Sokal: \textit{Some 
variants of the exponential formula, with application to the multivariate
Tutte-polynomial (alias Potts model)}, S\'eminaire Lotharingien
Combin. \textbf{61A} (2009), Article B61Ae, 33 pp. 


\bibitem{sok} A. D. Sokal: \textit{Bounds on the complex zeros of
  (di)chromatic polynomials and Potts-model partition functions},
  Combinatorics, Probability and Computing \textbf{10} (2001), No. 1, pp. 41-77 



\bibitem{tar} G. Tardos, private communication

\bibitem{wis} G. Wiseman: \textit{Set maps, umbral calculus, and the
chromatic polynomial}, Discrete Mathematics \textbf{308} (2008), No. 16,
pp. 3551-3564

\end{thebibliography}
\end{document}